\documentclass[11pt,a4paper]{article}
\usepackage[utf8]{inputenc}
\usepackage[english,francais]{babel}
\usepackage{amsmath}
\usepackage{amsfonts}
\usepackage{amssymb}
\usepackage{amsthm}
\usepackage{graphicx}
\usepackage[framemethod=TikZ]{mdframed}

\newtheorem{theorem}{Theorem} 
\newtheorem{lemma}[theorem]{Lemma} 
\newtheorem{proposition}[theorem]{Proposition} 
\newtheorem{remark}{Remark} 
\newtheorem{definition}[theorem]{Definition} 

\newtheorem{corollary}[theorem]{Corollary}

\usepackage[left = 100pt,right=100pt,top = 100pt,bottom = 100pt]{geometry}

\numberwithin{equation}{section}

\newcommand{\R}{\mathbb{R}}

\newcommand{\bH}{H}
\newcommand{\AC}{AC}
\newcommand{\bL}{L}
\newcommand{\bT}{\mathcal{T}}
\newcommand{\dist}{\mathrm{dist}}

\newcommand{\bv}{\mathbf{v}}
\newcommand{\eps}{\epsilon}

\newcommand{\Dx}{d_x}
\newcommand{\Dv}{d_v}
\newcommand{\Lip}{\mathrm{Lip}}

\DeclareMathOperator{\epi}{\mathbf{Epi}}
\DeclareMathOperator{\hypo}{\mathbf{Hypo}}
\DeclareMathOperator{\co}{\overline{\mathrm{co}}}

\usepackage{authblk}

\begin{document}
\selectlanguage{english}

\title{Uniqueness of the viscosity solution of a constrained Hamilton-Jacobi equation
}
\author{Vincent Calvez} 
\affil{Institut Camille Jordan, UMR 5208 CNRS \& Universit\'{e} Claude Bernard Lyon 1,  France\\
\texttt{vincent.calvez@math.cnrs.fr}}
\author{King-Yeung Lam}
\affil{Department of Mathematics, Columbus, OH, United States\\
\texttt{lam.184@math.ohio-state.edu}}
\date{}
\maketitle
\thispagestyle{empty}

\affil[$\star$]{}

\begin{abstract}
In quantitative genetics, viscosity solutions of Hamilton-Jacobi equations appear naturally in the asymptotic limit of selection-mutation models when the population variance vanishes. They  have to be solved together with an unknown function $I(t)$ that arises as the counterpart of a non-negativity constraint on the solution at each time. Although the uniqueness of viscosity solutions is known for many variants of Hamilton-Jacobi equations, the uniqueness for this particular type of constrained problem was not resolved, except in a few particular cases. Here, we provide a general answer to the uniqueness problem, based on three main assumptions: convexity of the Hamiltonian function $H(I,x,p)$ with respect to $p$, monotonicity of $H$ with respect to $I$, and $BV$ regularity of $I(t)$. 
\end{abstract}


\section{Introduction}

This note is intended to  address uniqueness of the viscosity solution of the following Hamilton-Jacobi equation, under a non-negativity constraint:
\begin{equation}\label{eq:HJ}
\begin{cases}
\partial_t u(t,x) + \bH(I(t),x,\Dx u(t,x)) = 0\, ,\quad & t\in (0,T)\, ,\;  x\in \R^d,\medskip\\
  \displaystyle\min_{x\in \R^d} u(t,x) = 0\,, & t \in (0,T),\\
u(0,x) = g(x)\,, \; & x \in \R^d,
\end{cases}
\end{equation}
where $\bH:   \R \times \R^d\times\R^d\to \R$ 
is of class $\mathcal C^2$, and the initial data $g\in W^{1,\infty}_{loc}(\R^d)$ satisfies $\min g = 0$. 
%
%
This problem arises naturally in the analysis of quantitative genetics model in the asymptotic regime of small variance \cite{DJMP05,Perthame-book, BP07,PB08,BMP,LMP11}. 

The main difficulty beyond the classical issue of weak solutions in the viscosity sense stems from the role played by the scalar quantity $I(t)$ which is subject to no equation, but is attached to the non-negativity constraint $\min u(t,\cdot) = 0$. Moreover, its regularity  in context is usually low, typically of bounded variation  \cite{Perthame-book,PB08,BMP}. Interestingly, we shall see that $BV$ seems to be the natural regularity ensuring uniqueness of the constrained problem.  
In order to give a sense to problem \eqref{eq:HJ} in such a setting, we use the theory of viscosity solutions for equations with a measurable dependence in time which was first studied by H. Ishii \cite{Ishii} and then by P.L. Lions and B. Perthame \cite{LP87}.


\subsection{Motivation and previous works}
\label{sec:motivation}
A special case of constrained Hamilton-Jacobi equation \eqref{eq:HJ} arises in the  asymptotic  limit of the following  quantitative genetics  model proposed in \cite{BP07,PB08} (see also \cite{Perthame-book,BMP,LMP11}):
%
\begin{equation}\label{eq:pop}
\eps \partial_t n_\epsilon  =  n_\epsilon  R( I_\epsilon(t),x) + \epsilon^2 \Delta n_\epsilon \, ,\quad t>0\, ,\; x\in \R^d\medskip \,,\quad 
I_\epsilon(t) = \int_{\mathbb{R}^d} \psi(x) n_\epsilon(t,x)\,dx,
\end{equation}
where $n_\epsilon(t,x)$ is the density of a population structured by a $d$-dimensional  phenotypical trait $x$.   The reproduction rate $R$ of a given individual  depends both on its trait $x$, and the environmental impact of the population $I_\eps$. Individuals may burden differently,  and that burden is weighted by the function $\psi$ which is  bounded below and above by positive constants. The key point is that $I_\eps$ is a scalar quantity, so that individuals    compete for a single resource.   As is natural for biological populations, the density-dependent feedback is negative, meaning that $R(I,x)$ is  decreasing with respect to $I$.

The Hopf-Cole transformation $u_\epsilon = -\epsilon \log n_\epsilon$ transforms \eqref{eq:pop} into the following equation:
\[
{\partial_t u_\epsilon}  + R( I_\epsilon(t),x) + |\Dx u_\eps|^2 = \epsilon \Delta u_\epsilon,
\]
which yields formally  \eqref{eq:HJ} in the  vanishing viscosity limit $\epsilon \to 0$, with $H(I,x,p) =   R(I,x) + |p|^2$:  
\begin{equation}\label{eq:HJa}
\begin{cases}
\partial_t u   + R(I(t),x) + |\Dx u|^2=0\, ,\quad & t\in (0,T)\, ,\; x\in \R^d,\medskip\\
\displaystyle\min_{x\in \R^d} u(t,x) = 0\,, & t \in (0,T),\\
u(0,x) = g(x)\,, & x \in \R^d .
\end{cases}
\end{equation}
In fact, locally uniform convergence  to a viscosity solution was established under suitable assumptions on $R$ and the initial data, but along subsequences $\eps_n\to 0$ \cite{PB08,BMP,LMP11}.   Therein,  compactness of $\{I_\epsilon\}$ usually follows from a uniform $BV$ estimate.
The constraint $\min u(t,\cdot) = 0$ can   then be derived from natural properties of the integral $I_\epsilon = \int \psi \exp(-u_\epsilon/\epsilon)\, dx$ being uniformly positive and bounded in~$\epsilon$,  as a consequence of the negative feedback of $I_\epsilon$ on  growth.  However, the convergence of $u_{\epsilon_n}\to u$ does not bring any direct information about the limit function $I$ in the limit $\eps_n \to 0$, except that the constraint $\min u(t,\cdot) = 0$ must be satisfied at any time.

The uniqueness of the limiting problem \eqref{eq:HJa}, if available, enables to obtain the convergence of the whole family of solutions $\{u_\epsilon\}$ as $\epsilon\to 0$. It interesting in the mathematical as well as biological perspectives, since the limiting problem \eqref{eq:HJa} determines much of the Darwinian  evolutionary  dynamics of the population model \eqref{eq:pop}. When $\epsilon$ is small, the population $n_\epsilon$ concentrates at the  point(s) where the limit function $u$ reaches its minimum value $0$.

The uniqueness  was first treated in \cite{PB08},  for the particular case when $R(I,x)$ is separable in the following sense: 
$$
R(I,x) = B(x) - D(x) Q(I), \quad \text{ or } \quad R(I,x) = B(x) Q(I) - D(x),
$$
with  positive functions $B,D$, and a monotonic function $Q$ 
such that $R$ is decreasing with respect to $I$. 

Later on, the uniqueness for \eqref{eq:HJa} was treated in \cite{RM17} under convexity assumptions on $R(I,x)$ and  the initial condition $g$, essentially: $R$ decreasing with respect to $I$, and concave with respect to $x$, plus $g$ convex. It was proved that convexity is propagated forward in time so that the solution $u(t,x)$ to \eqref{eq:HJa} is a solution in the classical $\mathcal C^1$ sense. Hence, it has always a unique minimum point $\overline{x}(t)$, which is a smooth function of $t$. As a result, $I(t)$  is necessarily smooth in that setting, since it can be determined by the implicit relation $R(I(t),\overline{x}(t)) = 0$. Biologically interpreted, their results describes very well the Darwinian dynamics of a monomorphic population as it evolves smoothly towards a (global) evolutionary attractor. Recently, the preprint \cite{Kim} tackled the uniqueness of \eqref{eq:HJa} when the trait space is one-dimensional, with a mixture of separable and non-separable growth rate $R(I,x)$, without convexity assumptions. However, the uniqueness result is also restricted to the case of continuous functions $I(t)$, which is not guaranteed in the absence of convexity. 

\subsection{Assumptions and main result}

In this paper, we establish uniqueness of solutions to problem \eqref{eq:HJ} under   mild conditions. In distinction with previous works, we assume  neither  (i) separability of the Hamiltonian $H(I,x,p)$ in any of its variables;   nor (ii) convexity in the trait variable $x \in \mathbb{R}^d$.   In particular, we can handle solutions $(u,I)$, allowing possibly: 
\begin{itemize}
\item $x \mapsto u(t,x)$ to possess multiple minimum points $\{\overline{x}_i(t)\}$; and
\item the Lagrange multiplier $I(t)$ to be discontinuous. 
\end{itemize}
Both of them are natural and attractive features of the solutions of population genetics models.

We restrict to $ \mathcal{C}^2$ Hamiltonian functions $H(I,x,p)$ which are convex and super-linear with respect to the third variable $p$:    
\begin{description}
\item[(H1):] \quad $\displaystyle  (\forall {I}, x,p)\quad  d^2_{p,p}H({I},x,p)>0\,, \quad (\forall {I}, x)\quad\lim_{|p|\to +\infty} \dfrac{H(I,x,p)}{|p|} = +\infty\, .$

\end{description}
Our uniqueness result strongly relies on the following monotonicity assumption: 
\begin{description}
\item[(H2):] \quad $\displaystyle  (\forall {I}, x,p)\quad    \dfrac {\partial H}{\partial I} ({I},x,p)<0\, .$
\end{description}

As we are dealing with convex Hamiltonians, it is approriate to reformulate the problem using suitable  representation formulas: For a given function $I(t)$, we define the {\em variational solution} $V(t,x)$ of \eqref{eq:HJ} as follows:
\begin{equation}\label{eq:Hopf Lax}
V(t,x) = \displaystyle \inf_{\left \{\begin{subarray}{c}  \gamma\in AC(0,t)\\ \gamma(t) = x\end{subarray}\right \}}\left\{ \int_0^t L({I}(s), \gamma(s),\dot \gamma(s))\, ds  + g(\gamma(0))\right\}\, ,
\end{equation}
for $(t,x) \in (0,T) \times \mathbb{R}^d$. 
Here $\bL: \R \times \R^d\times\R^d\to \R$ is the Lagrangian function, {\em i.e.} the  Legendre transform (or convex conjugate) of $\bH$ defined as: 
\begin{equation}\label{eq:legendre}
\bL({I}, x,v) = \sup_{p\in \R^d} \left \{p\cdot v - \bH({I}, x,p)\right \} \,.
\end{equation}
It is such that $d_p \bH$ and $\Dv \bL$ are reciprocal functions.  In this formulation, the problem \eqref{eq:HJ} becomes the determination of $I(t)$ so that the  value function $V(t,x)$ satisfies the constraint: $\min  V(t,\cdot) = 0$.
In the formulation \eqref{eq:Hopf Lax}, the role played by the scalar quantity ${I}(t)$ is perhaps more apparent:  it should somehow be adjusted in an infinitesimal and incremental way  to satisfy the following constraint at each time,   among all $\gamma \in AC(0,t)$ irrespective of the endpoint of $\gamma$:
\begin{equation}\label{eq:minimal}
(\forall t)\quad \inf_{\gamma\in AC(0,t)}\left\{ \int_0^t L({I}(s), \gamma(s),\dot \gamma(s))\, ds  + g(\gamma(0))\right\} = 0\, .
\end{equation}
Our methodology relies on the Lagrangian formulation  of the constraint \eqref{eq:minimal}. 
Due to the above variational reformulation, it is more appropriate to write the assumptions on the Lagrangian function: 
Assumptions {\rm (H1)}-{\rm (H2)} can be recast as: 
\begin{description}
\item[(L1):] \quad $(\forall {I}, x,v)$ \quad   $d^2_{v,v} \bL ({I},x,v) > 0 \,$.

\item[(L2):] \quad $(\forall {I}, x,v)$ \quad   $\dfrac {\partial \bL}{\partial I} ({I},x,v) > 0\,$.
\end{description}
%
%
We need two supplementary conditions, to be satisfied locally in $I\in (-J,J)$ for any constant $J>0$:
\begin{description}
\item[(L3):] There exist a constant $C_{\Theta}$, and a super-linear function $\Theta:\R_+\to \R_+$ so that
\begin{equation*}
(\forall {I}, x,v)\in (-J,J)\times \R^d \times \R^d \quad  L( {I}, x,v) \geq \Theta(|v|) -C_{\Theta}\, ,\quad \lim_{r\to +\infty} \dfrac{\Theta(r)}{r} = +\infty\, .
\end{equation*}
\item[(L4):]  For each $K >0$, there exists positive constants $\alpha_K, \beta_K$ such that
\begin{equation*}
(\forall I,x,v)\in (-J,J)\times B(0,K) \times \R^d \quad  |\Dx L( {I}, x,v)| \leq  \alpha_K  +  \beta_K L( {I}, x,v) \,.
\end{equation*}

\end{description}

Finally, we assume the initial data $g$ to be   locally Lipschitz continuous,  non-negative and coercive:
\begin{description}

\item[(G):] \quad $\displaystyle g \in W^{1,\infty}_{loc}(\mathbb{R}^d)\,, \quad   \min_{x\in \R^d} g(x) = 0\, , \quad \text{ and }\quad \lim_{|x| \to \infty} g(x) = +\infty.$

\end{description}
\begin{remark}\label{rem:A3}
The super-linearity in {\rm (L3)} holds true pointwise in $(I,x)$ by the very definition of the Legendre transform \eqref{eq:legendre}. The main point here is the uniformity with respect to $x$. 
\end{remark}

\begin{theorem}\label{th}
Under the assumptions {\rm (L1)}-{\rm (L4)}, suppose that ${I}_1$ and ${I}_2$ are two   non-negative  $BV$ functions associated with two variational solutions  $V_1$ and $V_2$  of \eqref{eq:Hopf Lax} with the same initial data  $g$ satisfying {\rm (G)}. Then, ${I}_1$ and  ${I}_2$ coincide almost everywhere, and so do $V_1$ and $V_2$. 
\end{theorem}

To make the connection with viscosity solutions, we  state the following  auxiliary result:
\begin{theorem}\label{prop:viscosity}
Under the assumptions {\rm (L1)}-{\rm (L4)}, given a function ${I}(t)$ of bounded variation, the variational solution $V(t,x)$ given by \eqref{eq:Hopf Lax} is the unique locally Lipschitz viscosity solution of \eqref{eq:HJ}  over $[0,T)\times \R^d$. 
\end{theorem}

The reason for separating these two results is to emphasize the use of the variational formulation in our proof. It would be of considerable interest to by-pass the variational formulation and derive uniqueness from PDE arguments only.    Uniqueness of unbounded solutions generally requires stringent conditions on the growth of the solution and the Hamiltonian \cite{Barles,CL87}, but here this issue is mediated by the fact that the Hamiltonian function is convex, and the solution is nonnegative by definition. We could not find a reference containing precisely Theorem \ref{prop:viscosity}, but \cite{dalmasofrankowska} is close, and we adapt their proof to our context in the Appendix.

\subsection{Examples}\label{sec:examples}

First, we apply our result to the special case presented in Section \ref{sec:motivation}. 

\begin{corollary} 
Consider the problem \eqref{eq:HJa}  with the initial data $g$ as in {\rm (G)}. Assume that 
$R  \in \mathcal{C}^2( \mathbb{R}_+ \times \mathbb{R}^d)$ satisfies 
\begin{equation}\label{eq:condition R}
(\forall {I}, x)\quad  \dfrac {\partial R}{\partial I} (I,x) <0 
\quad \text{and} \quad (\forall I)\quad \sup_{x\in \R^d} R(I,x) < +\infty\, ,
\end{equation}
%
then the solution pair $(u,I)$ to the constrained Hamilton-Jacobi equation \eqref{eq:HJa} is unique,  in the class of locally Lipschitz viscosity solutions $u$, and $BV$ functions $I$.
\end{corollary}

The second condition in \eqref{eq:condition R} is natural from the biological viewpoint as the net growth rate is presumably bounded from above. 
Since the three other conditions {\rm (L1)}, {\rm (L2)} and {\rm (L4)} are straightforward, 
it is sufficient to  verify {\rm (L3)}. Indeed, the Lagrangian $L$ is given by $\frac14 |v|^2 - R(I,x)$. Moreover,  since the range of the $BV$ function $I(t)$ lies in some bounded interval $[0,J]$, it is sufficient to restrict $R$ to $[0,J]\times \mathbb{R}^d$ to address uniqueness over a bounded time interval $(0,T)$. Since it is immediate from  \eqref{eq:condition R} that $R$ is uniformly bounded over $[0,J]\times \R^d$, {\rm (L3)} follows.

Our result also includes relevant examples that were not covered by the previous contributions, particularly non-separable Hamiltonian functions  $H(I,x,p)$. 
For instance, consider the following  quantitative genetics model:
\begin{equation*}
\eps \partial_t n_\epsilon = \int \dfrac{1}{\eps} K\left ( \dfrac{x-x'}{\epsilon} \right ) B(I_\epsilon(t),x') n_\epsilon(t,x') \, dx' -   n_\epsilon D(I_\eps (t),x),
\end{equation*}
where  $I_\eps(t)$ is the same as in \eqref{eq:pop}, and $K$ is a probability distribution function that encodes the mutational effects after reproduction: if the parent has trait $x'$, and gives birth at rate $B(I_\eps,x')$, the trait $x$ of the offspring is distributed following $K_\eps(x-x') = \frac{1}{\epsilon}K\left( \frac{x-x'}{\epsilon}\right)$. Assume that $K$ is symmetric, and has finite exponential moments, and denote by $\mathcal K$ its Laplace transform: $\mathcal K(p) = \int  K(z) \exp(p\cdot z)\, dz$. Then, the limiting problem as $\epsilon\to 0$ is  \eqref{eq:HJ} with the following Hamiltonian function  \cite{BMP}: 
\begin{equation}\label{eq:H exp}
H(I,x,p) = B(I,x) \mathcal{K}(p) - D(I,x)\,. 
\end{equation}


\begin{corollary}\label{cor:genetics}
Consider the problem \eqref{eq:HJ} with the Hamiltonian  \eqref{eq:H exp}, and the initial data $g$ as in {\rm (G)}. Assume that 
$\mathcal K$ is the Laplace transform of a symmetric p.d.f. with finite exponential moments, and that  $B,D \in \mathcal{C}^2( \mathbb{R}_+ \times \mathbb{R}^d)$ are non-negative functions which satisfy the following conditions: $B>0$ everywhere, and $(\forall I)$  $B(I,\cdot)\in L^\infty(\R^d)$, accompanied with the following monotonicity conditions:
\begin{equation*}
\\ 
(\forall {I}, x) \quad  \dfrac {\partial B}{\partial I} (I,x) <0,\, \text{ and }\,  \dfrac {\partial D}{\partial I}  (I,x) > 0\, ,
\end{equation*}
then the solution pair $(u,I)$ to the constrained Hamilton-Jacobi equation \eqref{eq:HJ} is unique,  in the class of locally Lipschitz viscosity solutions $u$, and $BV$ functions $I$.
\end{corollary}

\begin{proof}
There are a few items to check in order to apply Theorem \ref{th}. 
Firstly, the Hamiltonian function $H$ \eqref{eq:H exp} clearly verifies {\rm (H1)} and {\rm (H2)}, hence {\rm (L1)} and {\rm (L2)} follows. Secondly,  the Lagrangian function $L$ associated with the Hamiltonian $H$ \eqref{eq:H exp} is 
\begin{equation*}
L(I,x,v) = B(I,x)  \mathcal L\left ( \dfrac{v}{B(I,x)} \right ) + D(I,x)\, ,
\end{equation*}
where $\mathcal{L}(v)$ is the Legendre transform of $\mathcal{K}(p)$. Let $I_1,I_2$ be the two functions that are involved in the uniqueness test. Let  $[0,J]$ be a bounded interval in which both $I_1$ and $I_2$ take values. Condition {\rm (L3)} is clearly verified with $\Theta = 1 + \mathcal L$,  
because $L$ is decreasing with respect to the value $B$, $\sup_{[0, J]\times \R^d} B<+\infty$, and $\mathcal L$ satifies {\rm (L3)} automatically (see Remark \ref{rem:A3}).
The justification of {\rm (L4)} requires more work. 
We begin with the following inequality: 
\begin{equation}\label{eq:DL quad}
 \Dv\mathcal L(v) \cdot v \leq 2\left ( 1 + \mathcal L(v)\right )\, .
\end{equation}
To derive it, consider the following pointwise inequality: for all $X\in \R$, $\cosh X \leq 1 + \frac12 X \sinh X$, which in turn implies the following one by symmetry  of $K$:
\begin{multline*}
\mathcal K(p) =  \int K(z) \cosh(p\cdot z)\,dz \\\leq 1 + \frac12 \int K(z)  (p\cdot z) \sinh (p\cdot z)\, dz = 1 + \frac12 \int K(z)  (p\cdot z) \exp(p\cdot z)\, dz  \, .   
\end{multline*}
By applying this estimate to $ p = \Dv \mathcal L(v) \Leftrightarrow v = d_p \mathcal K(p)$ in the dual Legendre transformation $\mathcal{L}\leftrightarrow\mathcal{K}$, we deduce,  by the fact $\mathcal{K}(p) + \mathcal{L}(v) \geq p \cdot v$, that
\begin{equation*}
\Dv \mathcal L(v) \cdot v - \mathcal L (v)\leq \mathcal{K}(p)  \leq 1 + \frac12 \Dv \mathcal L(v) \cdot v\, . 
\end{equation*}
This yields the simple estimate announced in \eqref{eq:DL quad}.

The technical condition {\rm (L4)} is reformulated in this context as follows, after division by $B>0$: 
\begin{multline*}
\left | \Dx (\log B)    \mathcal L\left ( \dfrac{v}{B} \right ) - \Dx (\log B) \left ( \Dv \mathcal L\left ( \dfrac{v}{B} \right ) \cdot \dfrac{v}B \right ) + \dfrac{\Dx D}{B}  \right | 
%
%
%
\\  \alpha_K \dfrac{1}{B} +     
 \beta_K \left (\mathcal L\left ( \dfrac{v}{B } \right ) + \dfrac{D}{B} \right ) \, . 
\end{multline*}
%
It is indeed guaranteed for a suitable choices of $\alpha_K,\beta_K$. The main arguments besides \eqref{eq:DL quad} are: both $|\Dx (\log B)|$ and $|d_xD/B|$ are locally uniformly bounded from above,  $1/B$ is locally uniformly bounded from below, and $D$ is non-negative.
\end{proof}

\paragraph{Acknowledgment.} This work was initiated as the first author was visiting Ohio State University.  VC has funding from the European Research Council (ERC) under the European Union’s Horizon 2020 research and innovation programme (grant agreement No 639638). KYL is partially supported by the National Science Foundation under grant DMS-1411476.

\section{Regularity of the minimizing curves}

The purpose of this section is to establish $BV$ regularity  of the derivative $\dot \gamma$ of any minimizing curve $\gamma$ in \eqref{eq:Hopf Lax}. Such regularity is crucial in our argument of uniqueness. 
First, we establish the following  consequence of the convexity and the super-linearity of the Lagrangian: 

\begin{lemma}\label{lem:coercivity}
Assume {\rm (L1)} and {\rm (L3)}, then
\begin{equation*}
(\forall {I},x)\quad \lim_{|v|\to +\infty} \dfrac{\Dv L({I},x,v)\cdot v}{|v|} = +\infty\, ,
\end{equation*}
and the limit is uniform over $({I},x)$ lying in compact subsets of $\mathbb{R}^{d+1}$. 
\end{lemma}
\begin{proof}
Fix $K>0$  and let $(I,x) \in [-K,K]\times B(0,K)$.
Let $M>0$ be given, and choose $r_0= r_0(K)>0$ large enough so that for all $r \geq r_0$, 
\begin{equation*}
\Theta(r) - C_\Theta -  L({I},x,0)  \geq Mr \, . 
\end{equation*}
Then we have, for 
$r \geq r_0$ and $e \in \mathbb{S}^{d-1}$,
$$
L({I},x,re) - L({I},x,0) \geq Mr \, .
$$
By convexity of $L$ in $v$, we have, for all $e \in \mathbb{S}^{d-1}$ and $r\geq r_0$,
$$
\Dv L({I},x,re)\cdot e = \frac{d}{dr} \big(L({I},x,re)\big)  \geq \frac{L({I},x,re) - L({I},x,0)}{r} \geq M\, .
$$
Therefore, for each $M$ there exists $r_0$ such that 
$
\Dv L({I},x,v)\cdot v \geq M |v|$, provided that $|v| \geq r_0$ and $({I},x) \in [-K,K] \times B(0,K)
$.
\end{proof}

We will now establish the $BV$ estimate of $\dot\gamma^{t,x}$ for $(t,x) \in (0,T)\times \mathbb{R}^d$. 
For the remainder of this section, we fix $T>0$ and $J$ so that $\sup_{(0,T)} |I| \leq J$. For ease of notation, dependence of various constants on $T$ and $J$ will be omitted. 

\begin{lemma}\label{lem:regularity}
For each $K$, there exists a constant $C_{K}$ such that, uniformly for  $(t,x) \in (0,T) \times B(0,K)$,  any minimizing curve $\gamma^{t,x} \in \AC(0,t)$ associated with \eqref{eq:Hopf Lax} satisfies
\begin{equation}\label{eq:BVV_bound}
[\dot\gamma^{t,x}]_{BV(0,t)} \leq C_{K}\left ( t   + [I]_{BV(0,t)}\right )\,.
\end{equation}
\end{lemma}
\begin{proof}
The proof is   divided into three steps, wherein classical arguments are  recalled for the sake of completeness. For the sake of notation, we drop the superscript of $\gamma^{t,x}$, assuming that the pair $(t,x) \in (0,T) \times B(0,K)$ is fixed   throughout the proof.

\noindent\textbf{Step \#1: $L^\infty$ bound on $\gamma$.} 
We deduce immediately the following bound from {\rm (L3)}:
\begin{align}
\int_0^t \Theta(\dot \gamma(s))\, ds - Ct  + g(\gamma(0))  & \leq \int_0^t L({I}(s),\gamma(s), \dot{\gamma}(s))\,ds + g(\gamma(0))\nonumber\\
&\leq  \int_0^t L({I}(s),x,0)\, ds + g(x)\,,\label{eq:H1 bound 0}
\end{align}
from which we deduce a non-optimal $W^{1,1}$ estimate
\begin{align*}
 \int_0^t |\dot \gamma(s)|\, ds & \leq Ct +  \int_0^t L({I}(s),x,0)\, ds + g(x) - g(\gamma(0))   
\,  . 
\end{align*}
by using the super-linearity of $\Theta$ in a crude way, namely, $L  \geq |v| -C$.  Furthermore, we deduce from $g \geq 0$ that $\dot \gamma$ belongs to $L^1(0,t)$, so that 
$$
\|\gamma(s)\|_{L^\infty(0,t)} \leq C'_{K} \quad \text{ provided }(t,x) \in (0,T)\times B(0,K).
$$

%
%


Let  $A=A(g,K)$ be the (local) Lipschitz bound on $g$ in the ball  with radius $C'_{K}$. 
By updating the constant $C$, we can assume that $L( {I}, x,v) \geq (A + 1)|v| -C$. 
Back to \eqref{eq:H1 bound 0}, we deduce that 
\begin{equation*}
(A+1)\int_0^t \left| \dot \gamma(s)\right| \, ds  \leq C t + \int_0^t L({I}(s),x,0)\, ds + A \int_0^t \left| \dot \gamma(s)\right| \, ds \,,
\end{equation*} 
We obtain as a consequence the following updated estimate:
\begin{equation}\label{eq:average dot gamma}
\dfrac1t \int_0^t \left| \dot \gamma(s)\right| \, ds  \leq  C + \dfrac1t \int_0^t L({I}(s),x,0)\, ds \leq C'\, ,
\end{equation}
where the bound $C'$ is uniform for $(t,x)\in (0,T)\times B(0,K)$, and ${I}(s)$ taking values in $[-J,J]$.  

\noindent\textbf{Step \#2: $L^\infty$ bound on $\dot \gamma$.}
We deduce from \eqref{eq:average dot gamma} that there exists a subset $\mathcal S_{t,x}\subset(0,t)$ of positive measure,  such that for all $\hat{s} \in \mathcal{S}_{t,x}$,  $|\dot \gamma(\hat s)| \leq 2C'$.

Since $\gamma$ is a minimizing curve, it satisfies the following Euler-Lagrange condition in the distributional sense: 
\begin{equation}\label{eq:euler lagrange}
- \dfrac{d}{ds}\left ( \Dv L({I}(s), \gamma(s), \dot \gamma(s))  \right ) + \Dx L({I}(s), \gamma(s), \dot \gamma(s)) = 0\, .
\end{equation}

Let $\mathrm{Leb}(f)$ denote the set of Lebesgue points of the function $f$. Let $s_1$ and $s_2$ belong to $\mathrm{Leb}(I)\cap \mathrm{Leb}(\dot \gamma)$. Let $\{\rho_n\}$ be a family of mollifiers. We can test \eqref{eq:euler lagrange} against $\int_0^s (\rho_n(s'-s_2) - \rho_n(s'-s_1)\, ds'$. After integration by parts, we find that:
\begin{equation}\label{eq:EulLag1}
\left |\int_0^t \Dv L({I}(s), \gamma(s), \dot \gamma(s)) \left ( \rho_n(s-s_2) - \rho_n(s-s_1) \right )\, ds\right | \leq  \int_0^t \left |\Dx L({I}(s), \gamma(s), \dot \gamma(s))\right |\, ds \, , 
\end{equation}
where we have used that $|\int_0^s (\rho_n(s'-s_2) - \rho_n(s'-s_1)\, ds'| \leq 1$ for $n$ large enough. 
Using the definition of the Lebesgue points, we find that 
\begin{equation}\label{eq:EulLag2}
\lim_{n\to +\infty}\int_0^t \Dv L({I}(s), \gamma(s), \dot \gamma(s))  \rho_n(s-s_i)\, ds = \Dv L({I}(s_i), \gamma(s_i), \dot \gamma(s_i))\, .
\end{equation}
Then, we can specialize $s_1\in \mathrm{Leb}(I)\cap \mathrm{Leb}(\dot \gamma)\cap \mathcal S_{t,x}$ because the latter has positive measure. We deduce from \eqref{eq:EulLag1}--\eqref{eq:EulLag2}, and the definition of $\mathcal S_{t,x}$ that 
\begin{equation*}
\left | \Dv L({I}(s_2), \gamma(s_2), \dot \gamma(s_2)) \right | \leq C + \int_{0}^{t} \left |\Dx L({I}(s'), \gamma(s'), \dot \gamma(s'))\right |\, ds'\, ,
\end{equation*} 
for all $s_2  \in \mathrm{Leb}(I)\cap \mathrm{Leb}(\dot \gamma)$. Multiplying by the unit vector $\frac{\dot \gamma (s_2)}{|\dot \gamma(s_2)|}$, and using {\rm (L4)}, we find:
\begin{align*}
 \dfrac{\Dv L({I}(s_2), \gamma(s_2), \dot \gamma(s_2))\cdot \dot \gamma (s_2)}{|\dot \gamma(s_2)|}  
&\leq C + \int_{0}^{t} \left |\Dx L({I}(s'), \gamma(s'), \dot \gamma(s'))\right |\, ds' \\
&\leq C +     \int_0^t\left [\alpha_{K'}  + 
\beta_{K'} L({I}(s'), \gamma(s'), \dot \gamma(s'))  
\right]\, ds' 
\, ,
\end{align*}
where $K' =  \|\gamma\|_{L^\infty(0,t)}$. We deduce from  the uniform bound of $\|\gamma\|_{L^\infty(0,t)}$, and the 
minimizing property of $\gamma$ that the  right-hand-side is uniformly bounded for $(t,x,I)\in (0,T)\times B(0,K) \times [-J,J]$.   Hence, 
$$  \dfrac{\Dv L({I}(s), \gamma(s), \dot \gamma(s))\cdot \dot \gamma (s)}{|\dot \gamma(s)|}  \leq C \quad \text{ a.e.  }s \in (0,t),$$
and the boundedness of $\dot \gamma$ is a consequence of Lemma~\ref{lem:coercivity}.

\noindent\textbf{Step \#3: $BV$ bound on $\dot \gamma$.}
Back to \eqref{eq:euler lagrange}, we see that 
$$
p(s) = \Dv L({I}(s), \gamma(s), \dot \gamma(s))$$ is Lipschitz continuous as $I, \gamma, \dot \gamma \in L^\infty$, and $L$ is $\mathcal C^2$. By the Fenchel-Legendre duality, we have $$\dot\gamma(s) = d_p H(I(s), \gamma(s), p(s)).$$ Therefore, $\dot \gamma$ is  $BV$.

From the chain rule involving $BV$ functions (see \cite{ambrosio-dalmaso} and references therein, as well as \cite[Theorem 3.96]{ambrosio-book}), we deduce as a by-product that 
\[[\dot\gamma ]_{BV(0,t)}\leq \left [d_pH \right ]_{\Lip([-J,J]\times B(0,\|\gamma\|_\infty)\times B(0,\|\dot\gamma\|_\infty))}\left ([p]_{BV(0,t)} + [ \gamma ]_{BV(0,t)}  +  [{I}]_{BV(0,t)}\right )\, .\]
Both $p$ and $\gamma$ are Lipschitz continuous, so there exists a constant $C$ such that $\max([p]_{BV(0,t)},[\gamma]_{BV(0,t)})\leq C t$. Hence, we have obtained  \eqref{eq:BVV_bound}. 
\end{proof}

\section{Uniqueness of the variational solution (proof of Theorem \ref{th})}

Let ${I}_1,{I}_2\in BV(0,T)$. There exists a constant $J$ such that $I_1$ and $I_2$ takes value in $[-J,J]$.
Recall the definition of the variational solutions $V_1,V_2$:
\begin{equation*}
\begin{cases}
V_1(t,x) = \displaystyle \inf_{\gamma_1\in \AC(0,t): \gamma_1(t) = x}\left\{ \int_0^t L({I}_1(s), \gamma_1(s),\dot \gamma_1(s))\, ds  + g(\gamma_1(0))\right\}\, ,\medskip\\
V_2(t,x) = \displaystyle \inf_{\gamma_2\in \AC(0,t): \gamma_2(t) = x}\left\{ \int_0^t L({I}_2(s), \gamma_2(s), \dot \gamma_2(s)) ds  + g(\gamma_2(0))\right\}\, .
\end{cases}
\end{equation*}
\begin{lemma}\label{claim:A}
There exists $C>0$ such that for $j=1,2$,
$$
V_j(t,x) \geq \min\left\{ \frac{|x|}{2}, \min_{|x'| \geq |x|/2} g(x')\right\} - Ct\, .
$$
\end{lemma}
\begin{proof}
By definition of the variational solution and {\rm (L3)}, we have
\begin{align*}
V_j(t,x) &= \int_0^{t} L\left (I_j(s),\gamma_j^{t,x}(s),\dot{\gamma}_j^{t,x}(s)\right )\,ds + g\left (\gamma_j^{t,x}(0)\right )\\
&\geq \int_0^{t} \left (|\dot{\gamma}_j^{t,x}(s)| - C\right )\,ds  + g\left (\gamma_j^{t,x}(0)\right )\\
&\geq \left |x - \gamma_j^{t,x}(0)\right |- Ct + g\left (\gamma_j^{t,x}(0)\right )\\
&\geq \min \left \{|x|/2, \min_{|x'| \geq |x|/2} g(x')\right \}-C t
\end{align*}
where we used the fact that $\gamma^{t,x}(t) = x$ and that 
$
|x| \leq |x - \gamma^{t,x}(0)| + |\gamma^{t,x}(0)|
$, so that: either $|x- \gamma^{t,x}(0)| \geq |x|/2$ or $|\gamma^{t,x}(0)| \geq |x|/2$.
\end{proof}

By Lemma \ref{claim:A} and {\rm (G)}, we deduce that $x \mapsto V_j(t,x)$ attains   minimum in some bounded set, say $B(0,K)$, uniformly for $t \in (0,T)$. 

Let $x_1^t$ (resp. $x_2^t$) be some minimum point for $V_1(t,\cdot)$ (resp. $V_2(t,\cdot)$) -- this might not be unique -- and let $\gamma_1^t(s)$ (resp. $\gamma_2^t(s)$) be an optimal trajectory ending up at $x_1^t$ (resp. $x_2^t$). 
We deduce from Lemma \ref{lem:regularity} that $\dot \gamma_j^{t}$ lies in $BV$, uniformly with respect to $t\in (0,T)$: 
\begin{equation}\label{claim:boundedness}
\max_{j=1,2} \left [\dot\gamma^t_j\right]_{BV(0,t)} \leq C\,.
\end{equation}

%

The optimality of $\gamma_1^t$ and $\gamma_2^t$, together with the constraints $\min V_1(t,\cdot) = 0$ and $\min V_2(t,\cdot) = 0$ implies the following set of inequalities: 
\begin{align*}
0 = V_1(t,x_1^t) & =  \int_0^t L({I}_1(s), \gamma_1^t(s),\dot \gamma_1^t(s))\, ds  + g(\gamma_1^t(0))\nonumber\\
& \leq    \int_0^t L({I}_1(s), \gamma_2^t(s),\dot \gamma_2^t(s))\, ds  + g(\gamma_2^t(0)) \nonumber \\
& =   \int_0^t L({I}_2(s), \gamma_2^t(s), \dot \gamma_2^t(s)) ds  + g(\gamma_2^t(0))+ \int_0^t \phi_1(t,s) \left ( {I}_1(s) - {I}_2(s) \right )\,ds\, ,\nonumber\\
& =  - \int_0^t \phi_1(t,s) \left ( {I}_2(s) - {I}_1(s) \right )\, ds\,.
\end{align*}
i.e. 
\begin{equation}\label{eq:ineq1}
\int_0^t \phi_1(t,s) \left ( {I}_2(s) - {I}_1(s) \right )\, ds \leq 0\, ,
\end{equation}
where the positive weight $\phi_1$ is given by
\begin{equation}\label{eq:phi1}
\phi_1(t,s) = \int_0^1  \dfrac{\partial L}{\partial I}\left ( (1-\theta) {I}_1(s) + \theta{I}_2(s), \gamma_2^t(s), \dot\gamma_2^t(s) \right ) \, d\theta  \, .
\end{equation}
Similarly, by exchanging the roles of the two solutions, we obtain
\begin{equation}\label{eq:ineq2}
\int_0^t \phi_2(t,s) \left ( {I}_2(s) - {I}_1(s) \right )\, ds \geq 0\, ,
\end{equation}
where the positive weight $\phi_2$ is   given by
\begin{equation}\label{eq:phi2}
\phi_2(t,s) = \int_0^1  \dfrac{\partial L}{\partial I}\left ( (1-\theta) {I}_2(s) + \theta{I}_1(s), \gamma_1^t(s), \dot\gamma_1^t(s) \right ) \, d\theta   \, .
\end{equation}
By Assumption {\rm (L2)} and the uniform  boundedness of $( I_j(s), \gamma^t_j(s), \dot\gamma^t_j(s))$ (by \eqref{claim:boundedness}), there exists $\lambda>0$ such that: 
\begin{equation}\label{eq:beta bound}
(\forall t,s) \quad \min_{j=1,2}  \phi_j(t,s) \geq \lambda\,.
\end{equation}


%
%
%
%

Functions of bounded variations have left- and right-limits everywhere. 
Here,  we focus on the value of the right-limit at the origin.  This is expressed in the following statement.


\begin{lemma}\label{lem:phi_BV_tau}
Let $\phi_1, \phi_2$ be defined as in \eqref{eq:phi1} and \eqref{eq:phi2}. Then
\begin{equation*}
\lim_{t \to 0+} \left\{  \left[ \phi_1(t,\cdot)\right]_{BV(0,t)}  + 
 \left[ \phi_2(t,\cdot)\right]_{BV(0,t)} 
\right\}=0 .
\end{equation*}
\end{lemma}
\begin{proof}
Our first observation is that $BV$ regularity of $\{{I}_{j}\}_{j=1,2}$ implies the following smallness estimate: 
\begin{equation}
\label{eq:smallness phi}
\lim_{t \to 0+}  \left\{ 
 \left[ I_1\right]_{BV(0,t)}  +  \left[ I_2\right]_{BV(0,t)}  \right\}
=0.
\end{equation}
The important point here is that the left point of the interval is fixed to 0. The same conclusion would not be true if the interval $(0,t)$ would be replaced with $(-t,t)$ due to possible jump discontinuity at the origin.  To prove \eqref{eq:smallness phi}, let us decompose ${I}_1$, say, into a difference of non-decreasing functions ${I}_1 = {I}_1^+ - {I}_1^-$. Then, 
\begin{equation}\label{eq:sheng}
\left[ I_1\right]_{BV(0,t)}  \leq  I_1^+\big|_{0+}^t   + I_1^-\big|_{0+}^t 
%
\underset{t\to 0+}{\longrightarrow} 0\,,     
\end{equation}
simply because ${I}_1^+$ and ${I}_1^-$ have right limits at the origin. 

By \eqref{eq:BVV_bound},  
 we get that this vanishing limit can be extended to $\dot \gamma_{j}^{t}$ as well: 
\begin{equation}\label{eq:chong}
\left[ \dot \gamma^t_j \right]_{BV(0,t)}  \leq C \left( t + \left[ I_j\right]_{BV(0,t)} \right) 
%
 \underset{t\to 0+}{\longrightarrow} 0.
\end{equation}
Consequently, we are able to estimate 
$\left[ \phi_2(t, \cdot)\right]_{BV(0,t)}$ 
as follows. 
To keep the idea concise, we will compute the derivative $\frac{\partial \phi_2}{\partial s}$ of the $BV$ function $\phi_2(t, \cdot)$ in the sense of a finite measure on $(0,t)$, so that $[\phi_2(t,\cdot)]_{BV(0,t)} = \int_{(0,t)}  | \frac{\partial \phi_2}{\partial s}(t,s)  |\,ds$.   We shall adopt this convention for the remainder of the paper.
\begin{align*}
  \left|\dfrac{\partial \phi_2}{\partial s}(t,s)\right | 
&\leq   \int_0^1 \left |d^2_{I,I}  L(\Gamma(s)) \left( (1-\theta)\dot{I}_2(s)  +\theta \dot {I}_1(s) \right) \right| \,d\theta \\
&\quad  + \int_0^1 \left|d^2_{{I},x} L(\Gamma(s)) \dot \gamma_1^t(s) \right|   \, d\theta  
 +    \int_0^1 \left|d^2_{{I}, v} L(\Gamma(s)) \ddot{\gamma}_1^t(s)  \right | 
 \, d\theta \\
 &\leq C \left (   |\dot {I}_1(s)| +  |\dot {I}_2(s)| + |\dot \gamma_1^t(s) | + |\ddot \gamma_1^t(s) | \right )\, ,
\end{align*}
where we have used the shortcut notation  $\Gamma(s) = ((1-\theta){I}_2(s) + \theta {I}_1(s), \gamma^t_1(s), \dot\gamma^t_1(s))$. We may integrate the latter over the open interval $(0,t)$ to obtain
\begin{align*}
[\phi_2(t,\cdot)]_{BV(0,t)} 
&\leq C \left(  [I_1  ]_{BV(0,t)} + [I_2]_{BV(0,t)} + t  + \left [\dot\gamma_1^t\right ]_{BV(0,t)}  \right)\\
&\leq C \left( [I_1]_{BV(0,t)} + [I_2]_{BV(0,t)} + t\right)
\end{align*}
where we used \eqref{eq:chong}. By \eqref{eq:sheng}, we  deduce that $\lim_{t \to 0+} [\phi_2(t,\cdot)]_{BV(0,t)} = 0$. The proof for $\phi_1$ is analogous. 
\end{proof}


We are now in position to prove Theorem \ref{th}.
\begin{proof}[Proof of Theorem \ref{th}]
Let $\mu  = {I}_2  - {I}_1$ and suppose to the contrary that  $\mu  \neq 0$ on a set of positive measure in $(0,T)$. 

We claim that we may assume, without loss of generality, that
$\mu \neq 0$ in a set of positive measure in $(0,t)$, for each $t \in (0,T)$. To see this claim, let
%
%
%
%
$$
t_0:= \sup\{t\geq 0: \mu(s) = 0 \,\, \text{ a.e. in }\, (0,t)\}.
$$
If $t_0 = 0$, we are done. If $t_0 >0$, then the variational solutions $V_1(t,x) \equiv V_2(t,x)$ are identical for $(t,x) \in [0,t_0] \times \mathbb{R}^d$. By Lemma \ref{claim:A}, $x\mapsto V_j(t_0,x)$ is coercive, and the other conditions in {\rm (G)} are clearly satisfied. Thus we may re-label the initial time to be $t_0$.  In any case, it suffices to derive a contradiction assuming $\mu \neq 0$ in a set of positive measure in $(0,t)$, for each $t \in (0,T)$. 

 
Using $ \int_{0}^t  \phi_2(t,s) \mu (s)\,ds  \geq 0$ (by  \eqref{eq:ineq2}), we may integrate by parts to obtain
\begin{align*}
\phi_2(t,t-) \int_0^t \mu(\tau)\,d\tau & = \int_{(0,t)} \frac{\partial \phi_2}{\partial s}(t,s) \left( \int_0^s \mu(\tau)\,d\tau\right)\,ds + \int_0^t  \phi_2(t,s) \mu (s)\,ds \\
&\geq \int_{(0,t)} \frac{\partial \phi_2}{\partial s}(t,s) \left( \int_0^s \mu(\tau)\,d\tau\right)\,ds 
\end{align*}
Taking the negative part, we deduce the following partial estimate,
\begin{equation}\label{eq:A1}
\phi_2(t,t-) \left ( \int_0^t \mu (s)\,ds \right)_- \leq   \left( \sup_{s\in (0, t)}\left| \int_0^s \mu(\tau)\,d\tau\right| \right)  \int_{(0,t)} \left| \frac{\partial \phi_2}{\partial s} (t,s)\right|\,ds
\end{equation}
Similarly we deduce from  \eqref{eq:ineq1} that
\begin{align*}
\phi_1(t,t-) \int_0^t \mu(\tau)\,d\tau & = \int_0^t \frac{\partial \phi_1}{\partial s}(t,s) \left( \int_0^s \mu(\tau)\,d\tau\right)\,ds +\int_0^t  \phi_1(t,s) \mu (s)\,ds \\
&\leq \int_0^t \frac{\partial \phi_1}{\partial s}(t,s) \left( \int_0^s \mu(\tau)\,d\tau\right)\,ds 
\end{align*}
Taking the positive part, we deduce the following complementary estimate,
\begin{equation}\label{eq:A2}
\phi_1(t,t-)  \left( \int_0^t  \mu (s)\,ds  \right)_+ \leq   \left( \sup_{s\in (0, t)}\left| \int_0^s \mu(\tau)\,d\tau\right| \right)  \int_0^t \left| \frac{\partial \phi_1}{\partial s} (t,s)\right|\,ds
\end{equation}
Combining \eqref{eq:A1} and \eqref{eq:A2}, together with \eqref{eq:beta bound}, we obtain 
\begin{equation}\label{eq:A3}
\lambda \left|  \int_0^t \mu (s)\,ds \right| \leq   \max_{i=1,2} \left\{ \int_0^t \left| \frac{\partial \phi_i}{\partial s} (t,s)\right|\,ds\right\}    \left( \sup_{s\in (0, t)}\left| \int_0^s \mu(\tau)\,d\tau\right| \right) 
\end{equation}
Next, Lemma \ref{lem:phi_BV_tau} ensures that there exists  $t_1>0$ so that 
\[
\sup_{t\in (0,t_1)}  \max_{i=1,2} \left\{ \int_0^t \left| \frac{\partial \phi_i}{\partial s} (t,s)\right|\,ds\right\} \leq \frac{\lambda}{2}. 
\]
Then, taking supremum in \eqref{eq:A3} for $0 < t < t_1$,  we have
\begin{align*}
\lambda \left(\sup_{t\in (0,t_1)}  \left|  \int_0^t \mu (s)\,ds \right| \right) \leq 
\frac{\lambda}{2}    \left(\sup_{s\in (0,t_1)}\left| \int_0^s \mu(\tau)\,d\tau\right| \right). 
\end{align*}
This implies $\int_0^t \mu(s)\,ds = 0$ for all $t \in  [0,t_1]$. Hence, $\mu(t) = 0$ almost everywhere on $(0,t_1)$. This is in contradiction with the assumption that $\mu \neq 0$ on a set of positive measure in $(0,t_1)$, and we conclude that ${I}_2 - {I}_1 = \mu = 0$ a.e.  Finally $V_1 \equiv V_2$ by virtue of the variational formulation.
\end{proof}


\section{The Pessimization Principle: $I(t)$ is non-decreasing}

The pessimization principle \cite{Diekmann} is a concept in adaptive dynamics, which says that if the environmental feedback is encoded by a scalar quantity $I(t)\in \R$ at any time, mutations and natural selection inevitably lead to deterioration/Verelendung. In the setting of this paper, it can be formulated by claiming that the population burden $I(t)$ is a non-decreasing function. 

In this section, we   give an additional assumption that guarantees this claim.
\begin{theorem}\label{thm:pess}
Under the assumptions {\rm (L1) - (L4)}, let $(u,I)$ be the unique solution pair $(u,I)$ to \eqref{eq:HJ}. Assume, in addition, that 
\begin{description}
\item[(L5):]\quad  $(\forall {I}, x,v)$  \quad $\displaystyle L(I,x,v) \geq L(I,x,0)$.
\end{description}
Then  $I$ is non-decreasing with respect to time. 
\end{theorem}

Note that {\rm (L5)} is equivalent to $d_vL(I,x,0) = 0$ by convexity, and thus to $d_p H(I,x,0) =0$ by duality. It is clearly verified for the two examples in Section \ref{sec:examples} by symmetry of the Hamiltonian function with respect to $p=0$.

\begin{corollary}
Under the assumption \eqref{eq:condition R}, let 
 $(u,I)$ be the unique solution pair to \eqref{eq:HJa}, with the initial data $g$ as in {\rm (G)}. Then $I$ is non-decreasing  with respect to time.
\end{corollary}

\begin{corollary}
Under the assumptions of Corollary \ref{cor:genetics}, let $(u,I)$ be the unique solution pair to the problem \eqref{eq:HJ} with the Hamiltonian  \eqref{eq:H exp}, and the initial data $g$ as in {\rm (G)}.  Then $I$ is non-decreasing  with respect to time.
\end{corollary}

\begin{proof}[Proof of Theorem \ref{thm:pess}]

We start by choosing the right-continuous representative of $I$ without loss of generality. For each $t>0$, let $x^t$ be a minimum point of $x\mapsto u(t,x)$ as before, and let
$\gamma^t$ be an associated minimizing curve ending up at $x^t$.

\noindent\textbf{Step \#1: $L(I(t-),x^t,0) \leq 0$ for all $t$.}
It follows from the non-negativity constraint and the dynamic programming principle that
$$
0 \leq u(s, \gamma^t(s))= \int_0^s L(I(s'),\gamma^t(s'), \dot\gamma^t(s'))\,ds' + g(\gamma^t(0)) \quad \text{ for }0 < s < t,
$$
and the equality holds when $s = t$. Hence, we deduce that $L(I(t-), \gamma^t(t-), \dot\gamma^t(t-)) \leq 0.$ Since $\gamma^t(t-) = \gamma^t(t) = x^t$, we may use {\rm (L5)} to deduce that
$$
L(I(t-), x^t,0) \leq L(I(t-),\gamma^t(t-), \dot\gamma^t(t-)) \leq 0.
$$

\noindent\textbf{Step \#2: $I(t-) \leq I(t+)$ for all $t$.}
%
Fix $t>0$, let $x^t$ and $\gamma^t$ be as above. We define $\gamma_1: [0,t+1]\to \mathbb{R}^d$ by
$$
\gamma_1(s) = \left\{
\begin{array}{ll}
\gamma^t(s) & \text{ for } 0 \leq s \leq t,\\
x^t  & \text{ for } s >t.  
\end{array}
\right.
$$
Then $\gamma_1 \in AC[0,t+1]$ and for all $0 <h< 1$,
$$
0 \leq  u(t+h, \gamma_1(t+h)) \leq \int_0^{t+h} L(I(s'), \gamma_1(s'),\dot\gamma_1(s'))\,ds' + g(\gamma_1(0)).
$$
Since $0=u(t, \gamma_1(t))=\int_0^{t} L(I(s'), \gamma_1(s'),\dot\gamma_1(s'))\,ds' + g(\gamma_1(0))$, we have
\begin{equation}\label{eq:5.10}
0 \leq \int_t^{t+h} L(I(s'), \gamma_1(s'),\dot\gamma_1(s'))\,ds' \leq \int_t^{t+h} L(I(s'),x^t,0)\,ds'.
\end{equation}
Dividing by $h$, and letting $h \to 0+$, we obtain $L(I(t+), x^t, 0) \geq 0$. Comparing with $L(I(t-),x^t,0) \leq 0$ (by Step \#1), we deduce from the monotonicity of $L$ in $I$ {\rm (L2)} that $I(t-) \leq I(t+)$ for all $t>0$. 

\noindent\textbf{Step \#3: Conclusion.}
%
%
Suppose to the contrary that $I(t_2) < I(t_1)$ for some $t_1 < t_2$. Since $I$ is right-continuous, there exists $t_3>t_2$ such that $I(t)< I(t_1)$ for all $t \in [t_2, t_3)$. 
Let $t_0= \sup\{t \in [t_1, t_3):  I(t_1) \leq I(t)\}$. 
Then $t_0 \leq t_2 < t_3$, and  
\begin{equation*}\label{eq:5.11}
I(t) < I(t_1) \leq I(t_0-)\quad \text{ for }t \in (t_0, t_3).
\end{equation*}
Now, using Step \#1 and \eqref{eq:5.10} from Step \#2, we have
$$
0 \leq \int_{t_0}^{t_3} \left[ L(I(s'), x^{t_0},0) - L(I(t_0-),x^{t_0},0)\right]\,ds'.
$$
However, this is in contradiction with \eqref{eq:5.11}, in view of the fact that $L$ is strictly increasing in $I$ (L2).
\end{proof}

\appendix

\section{Variational and viscosity solutions coincide (proof of Theorem \ref{prop:viscosity})}

Given $I \in BV(0,T)$, let
 $V(t,x)$ denote the  corresponding variational solution of \eqref{eq:Hopf Lax}, and let $u$ denote a locally Lipschitz viscosity solution of \eqref{eq:HJ}. The purpose of this section is to show that $u \equiv V$.

As the Hamiltonian is convex with respect to $p$, 
 sub-solutions in the almost everywhere sense, and viscosity sub-solutions in particular, lie automatically below the variational solution \cite{Barles,Fathi}. We include a proof here for the sake of completeness.
\begin{proposition}
Assume that $u$ is locally Lipschitz,  $u(0,x) \leq g(x)$ for all $x$, and that the following inequality holds for almost every $(t,x)\in (0,T)\times \R^d$,
\begin{equation}\label{eq:ae}
\partial_t u(t,x) + \bH({I}(t),x,\Dx u(t,x))  \leq 0\quad a.e.  
\end{equation}
Then, $u\leq V$. 
\end{proposition}
\begin{proof}
The proof is adapted from \cite[Section 4.2]{Fathi}. A more direct proof can be found in \cite[Section 9]{Barles} but the latter  assumes time continuity for $H$, which does not hold in the present case. A first observation is that \eqref{eq:ae} makes perfect sense as $u$ is differentiable almost everywhere by Rademacher's theorem. We shall establish that
\begin{equation}\label{eq:dominated}
u(t_2,\gamma(t_2)) - u(t_1,(\gamma(t_1)) \leq \int_{t_1}^{t_2} L({I}(s), \gamma(s), \dot \gamma(s))\, ds\, ,
\end{equation} 
for all  curves $\gamma\in W^{1,\infty}$. Thus, the result will follow immediately by taking the infimum with respect to $\gamma$, and invoking regularity of minimizing curves, as in Lemma \ref{lem:regularity}. 

To prove \eqref{eq:dominated}, we proceed by a density argument. 
The case of a linear curve $\gamma = x  + (s-t_1) v$   is handled as follows: firstly, we deduce from \eqref{eq:ae} that 
\begin{equation}\label{eq:ae3}
\partial_t u(t,x) + \Dx u(t,x)\cdot  v \leq L({I}(t), x, v)\quad a.e.  
\end{equation}
Secondly,  by Fubini's theorem  one can find a sequence $x_n\to x$ such that \eqref{eq:ae3} holds almost everywhere in the line $\{ (s,x_n+(s-t_1) v) \}$ for each $n$. 
Therefore, we can apply the chain rule to $u(s,x_n+(s-t_1) v)$, so as to obtain:
\begin{equation}\label{eq:ae4}
\dfrac{d}{ds}\left ( u(s,x_n+(s-t_1) v)\right ) \leq L({I}(s), x_n+(s-t_1) v,v)\quad a.e.
\end{equation}
We deduce that \eqref{eq:dominated} holds true for all linear curves by integrating \eqref{eq:ae4} from $t_1$ to $t_2$ and taking the limit $n\to +\infty$. 

Consequently, \eqref{eq:dominated} holds true for any piecewise linear curve. The conclusion follows by a density argument of piecewise linear curves in the set of curves having   bounded measurable derivatives. 
\end{proof}

It remains to show that viscosity super-solutions lie above the variational solution. 
The criterion for super-solution   for time-measurable Hamiltonians that we adopt is the following one. (See  \cite{Ishii,LP87} for various other equivalent definitions.)
\begin{definition}[Viscosity super-solution]\label{def:supersol}
Let $\phi\in \mathcal C^1(\R^d)$ be such that the minima of $u(t,\cdot) - \phi$ are reached in a ball of radius $R$ for all $t\in [0,T]$. Let $\mathfrak{M}(t)$  be the set of minimum points of $u(t,\cdot) - \phi$, and $m(t) = \min u(t,\cdot) - \phi$. Then, it is required that the following inequality holds true in the distributional sense: 
\begin{equation}\label{eq:supersol}
m'(t) + \underset{y\in \mathfrak{M}(t)}{\sup} H({I}(t), y,  d_x \phi(y)) \geq 0\quad \text{in}\; \mathcal{D}'(0,T)\, .
\end{equation} 
\end{definition}

\begin{proposition}
Assume that $u$ is a locally Lipschitz viscosity super-solution, in the sense of Definition \ref{def:supersol},  and that $u(0,x) \geq g(x)$ for all $x$. Then, $u\geq V$. 
\end{proposition}

\begin{proof}
We follow the lines of \cite{dalmasofrankowska} which is essentially based on convex analysis. We adapt their proof in our context for the sake of completeness.  We will first prove the proposition in the special case of $I \in W^{1,\infty}(0,T)$. This assumption will be relaxed to $I \in BV(0,T)$ at the end of the proof.

\noindent\textbf{Step \#1: Finding the backward velocity: setting of the problem.}
The key is to find, for each $(t,x)$, a particular direction $\bv(t,x)$, such that the following inequality holds true:
\begin{equation}\label{eq:graal}
d_+ u(t,x)(1,\bv(t,x)) \geq L(I(t),x,\bv(t,x))\, ,
\end{equation}
where $d_+ u(t,x)(\mu,v)$ is the one-sided directional differentiation in the direction  $(\mu,v)$:
\begin{equation*}
d_+ u(t,x)(\mu,v) = \limsup_{s\to 0+} \dfrac{u(t,x) - u(t-s\mu,x-sv)}{s}\, .
\end{equation*}


We can interpret \eqref{eq:graal} as follows: there exists an element which is common to the partial epigraph of $v \mapsto L(I(t),x, v)$:
$$
\mathcal E_{t,x} = \epi_v (L(I(t),x,v)) = \left\{(v,\ell) \in \mathbb{R}^{d}\times \mathbb{R}: r \geq L(I(t),x,v) \right\}, 
$$and to the hypograph of $v \mapsto d_+ u(t,x)(1,v)$:
$$
\hypo_v(d_+ u(t,x)(1,v)) = \left\{ (v,\ell) \in \mathbb{R}^{d}\times \mathbb{R}: r \leq d_+ u(t,x)(1,v) \right\}.
$$ 
For technical reason, we consider the full hypograph of $d_+u(t,x)$, taken with respect to variables $(\mu,v) \in \mathbb{R} \times \mathbb{R}^d$:
\begin{equation}\label{eq:hypo tx}
\mathcal H_{t,x} = \hypo_{(\mu,v)}(d_+u(t,x)(\mu,v))= \hypo_{(\mu,v)} \left ( \limsup_{s\to 0+} \dfrac{u(t,x) - u(t-s\mu ,x-sv)}{s} \right )\, .
\end{equation}
In contrast with $\hypo_v(d_+ u(t,x)(1,v))$, $\mathcal{H}_{t,x}$ is a cone because the quantity in \eqref{eq:hypo tx} is positively homogeneous with respect to $(\mu,v)$. In fact, it coincides with the definition of a contingent cone, up to a change of sign. If $S\subset \R^{N}$ is a non-empty subset, and  $z\subset \R^{N}$, recall that the contingent cone of $S$ at $z$, denoted by  $\bT_S (z)$, is defined as follows \cite[Definition 3.2.1]{Aubin}:
\begin{equation*}
w\in \bT_S (z) \quad \Longleftrightarrow \quad \liminf_{s\to 0^+} \dfrac{\dist(z + s w,S)}{s} = 0\, .
\end{equation*}
Then, we claim the following equivalence: 
\begin{equation}\label{eq:equivalence}
\mathcal H_{t,x} = - \bT_{\epi u} (t,x,u(t,x))\, .
\end{equation}
For the convenience of readers, the equivalence \eqref{eq:equivalence} is illustrated in Figure \ref{fig:cotingentcones} for a scalar function $u$.
Now we show \eqref{eq:equivalence}. Indeed, $(\mu,v,\ell)$ belongs to $ - \bT_{\epi u}(t,x,u(t,x))$ if and only if there exist subsequences  $s_n\to 0+$ and $(t_n,x_n,u_n)$ such that:
\begin{equation*}
\begin{cases}
t - s_n \mu  = t_n +  o(s_n)\\
x - s_n  v = x_n + o(s_n)\\
u(t,x) - s_n \ell  = u_n + o(s_n)
\end{cases}\, , \quad \text{and}\quad 
u_n \geq u(t_n, x_n)\, .
\end{equation*}
The latter inequality is inherited from the choice $S = \epi u$. Reorganizing the terms, and using the Lipschitz continuity of $u$, we obtain:
\begin{align*}
&u(t,x) - s_n\ell  \geq u(t-s_n \mu,x-s_nv) + o(s_n)\, , \\
& \dfrac{u(t,x) - u(t-s_n \mu,x-s_nv)}{s} \geq \ell + o(1)\, .
\end{align*} 
The latter is precisely \eqref{eq:hypo tx}.   i.e. $(\mu,v, \ell) \in \mathcal{H}_{t,x}$ and this proves \eqref{eq:equivalence}.

\begin{figure}
\begin{center}
\begin{minipage}{.58\linewidth}
\includegraphics[width=\linewidth]{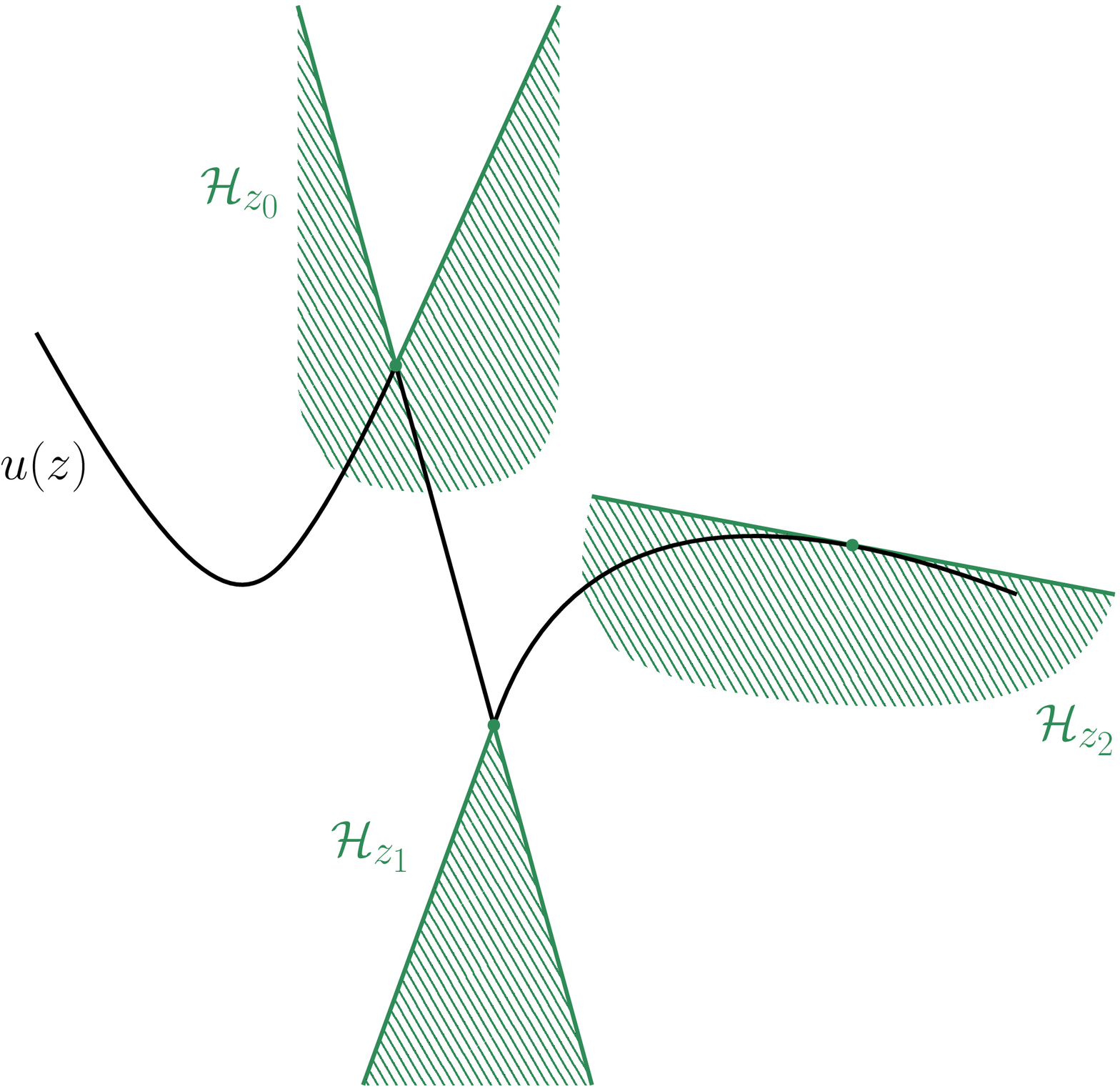} 
\end{minipage}\qquad
\begin{minipage}{.18\linewidth}
\begin{center}
\includegraphics[width=.9\linewidth]{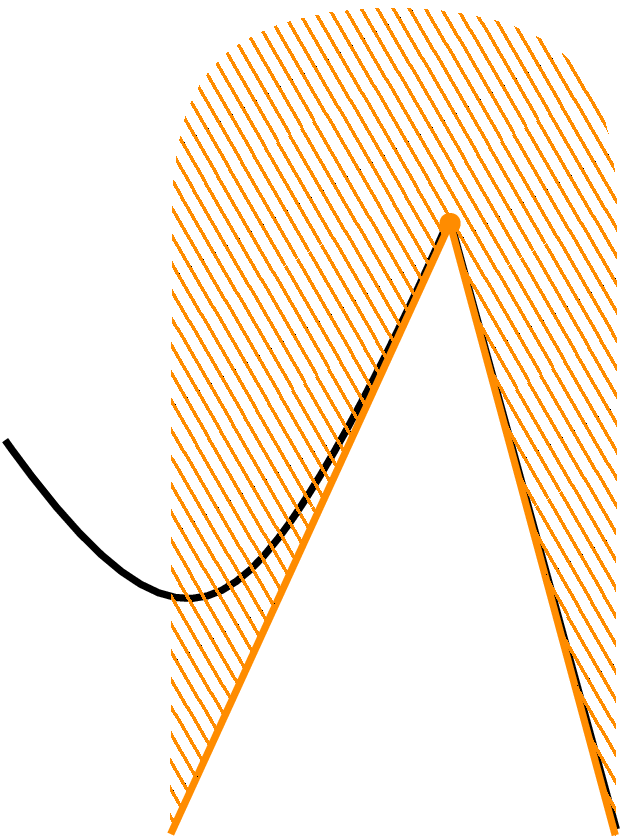}\vspace{10pt}\\ 
\includegraphics[width=.8\linewidth]{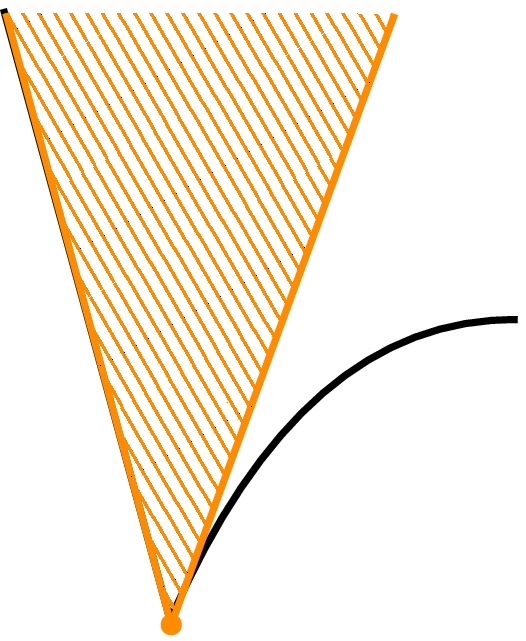}\vspace{10pt}\\
\includegraphics[width=1.4\linewidth]{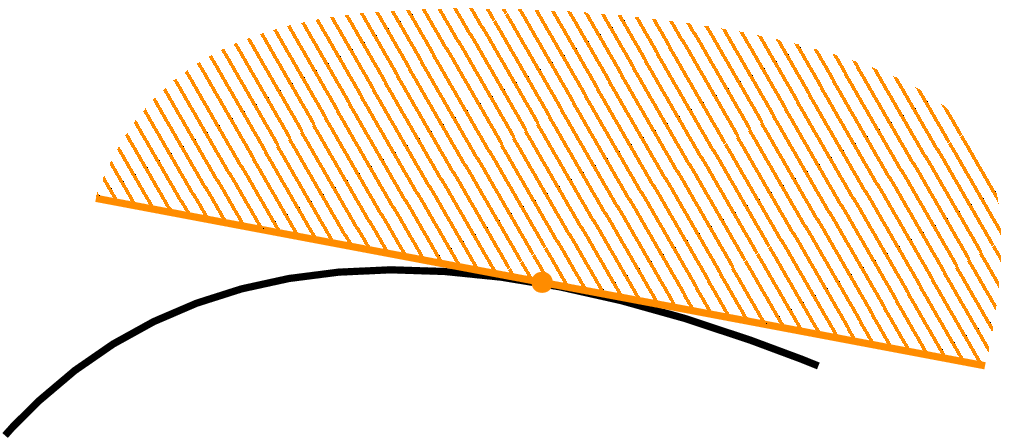} 
\end{center}
\end{minipage}
\caption{Illustration of various shapes of cones as they may appear for a scalar function $u$ (in opposition to the text where  the domain of $u$ is genuinely multi-dimensional). The set $\mathcal H$ is represented in shaded green, whereas the corresponding contingent cones $\bT = - \mathcal H$ are depicted in shaded orange.\label{fig:cotingentcones}}
\end{center}
\end{figure}

%

Summarizing, we are seeking an element $(1,v)$ which is common to $\mathcal H_{t,x}$ and to $\{1\}\times \mathcal E_{t,x}$. 
The latter is a convex set, but the former is not necessarily convex. Therefore, we are led to consider its convex closure $\co(\mathcal H_{t,x})$ in order to use the separation theorem. Next, we shall use the viability theory to remove the convex closure, exactly as in \cite{dalmasofrankowska}.  

\noindent\textbf{Step \#3: Finding the backward velocity: the separation theorem.}
We wish to avoid separation of the two convex sets $\co(\mathcal H_{t,x})$ and $\{1\}\times \mathcal E_{t,x}$. We argue by contradiction. If the two sets are separated, then there exists a linear form $q\cdot  + \langle p, \cdot \rangle$ such that (i) $\co(\mathcal H_{t,x})$ lies below the hyper-plane $\{(\mu,v,\ell): \ell =  q\mu+ \langle p, v \rangle \}$, and (ii) $\{1\}\times \mathcal E_{t,x}$ lies strictly above it \cite{Rockafellar}. We deduce from the latter condition (ii) that $q + \langle p, v \rangle \leq L(I(t),x,v) - \delta$ for all $v\in \R^d$ and some $\delta>0$. This can be recast as $q + H(I(t),x,p) \leq -\delta$ from the definition of the Legendre transform. On the other hand, we deduce from condition (i) that  
\begin{equation*}
\limsup_{s\to 0^+}\dfrac{u(t,x) - u(t-s\mu ,x-sv)}{s}  = d_+u(t,x)(\mu,v) \leq q\mu+ \langle p, v \rangle \, , 
\end{equation*}
{for all }$(\mu,v)\in \R^{d+1}$. Consequently, $(q,p)$ belongs to the subdifferential of $u$ at $(t,x)$. By applying the usual criterion of viscosity super-solutions (for continuous Hamiltonian functions),  we find that $q + H(I(t),x,p) \geq 0$. This is a contradiction. Thus, the two convex sets are not separated, i.e.  
\begin{equation}\label{eq:non empty coH}
(\forall t,x) \quad  \co(\mathcal H_{t,x})\cap \left (  \{1\}\times \mathcal E_{t,x}\right )\neq \emptyset\, .
\end{equation}

\noindent\textbf{Step \#4: Finding the backward velocity: the viability theorem.} 
Note that \eqref{eq:non empty coH} is equivalent to 
\begin{equation}\label{eq:non empty coH2}
(\forall t,x) \quad  \co(-\mathcal{T}_{\epi u}(t,x,u(t,x)))\cap \left (  \{1\}\times \mathcal E_{t,x}\right )\neq \emptyset\, .
\end{equation}
We wish to use the viability theorem \cite[p. 85]{Aubin} (see also \cite[Theorem 2.3]{dalmasofrankowska}):
\begin{theorem}[Viability]\label{th:viability}
Suppose that $G:\R^N \rightsquigarrow\R^N$ is an upper semi-continuous set-valued map with compact convex values. Then for each closed set $S\subset \R^N$, the following statements are equivalent:
\begin{itemize}

\item[\rm (a)]  \quad $(\forall z\in S)   \quad \bT_S (z) \cap G(z) \neq \emptyset$;

\item[\rm (b)] \quad $(\forall z\in S)   \quad \left ( \co \bT_S (z)\right ) \cap G(z) \neq \emptyset$.
\end{itemize}
%
\end{theorem}

Further compactness estimate is required in order to apply Theorem \ref{th:viability}. 
We claim that we can restrict \eqref{eq:non empty coH} to a compact set: 
\begin{equation*}
\co(\mathcal H_{t,x}) \cap \left ( \{1\}\times \mathcal E_{t,x} \right )\cap \left ( \{1\}\times B(0,R_{ |x|})\times \left [m,M\right ]\right )  \neq \emptyset\,,
\end{equation*}
where for each $K>0$,  $R_{K} = \max\{1, r_{K}\}$, with $r_K$ is increasing in $K$ such that
\begin{equation}\label{eq:rk}
\Theta (r) > [u]_{\Lip(\overline{J \times B(0,K)})} (1 + r) + C_\Theta \quad \text{ for all }r \geq r_K  
\end{equation}
(the choice of $r_K$ is possible due to the superlinear growth of $\Theta$), and $m,M$ are respectively $m = \min L$, $M = \max L$ where both minimum and maximum are taken over the set $J\times  \{x\}\times B(0,R_{|x|})$.

To this end, 
consider the following two options: either the dual cone $(\mathcal H_{t,x})^-$ is empty or non-empty. In the first case, it implies $\co (\mathcal H_{t,x}) = \R^d$, so that any element of  $\mathcal E_{t,x}$ is appropriate. In this case we have
\begin{equation}\label{eq:choice1}
(1, 0, L(I(t),x,0)) \in \co(\mathcal H_{t,x}) \cap \left ( \{1\}\times \mathcal E_{t,x} \right )\cap \left ( \{1\}\times B(0,1)\times \left [m,M\right ]\right ) \,.
\end{equation}
In the second case, $(\mathcal H_{t,x})^-$ is non-empty. Hence, there exists a linear form $q\cdot  + \langle p, \cdot \rangle$ such that  $\co(\mathcal H_{t,x})$ lies below the linear set $\{ q\mu+ \langle p, v \rangle \}$ as in  Step \#3. 
 Therefore, every common point $(1,v,\ell) \in \overline{\rm co}(\mathcal{H}_{t,x}) \cap (\{1\}\times \mathcal{E}_{t,x})$ (and there is at least one such point) must satisfy 
$$
 L(I(t),x,v) \leq \ell \leq q + \langle p, v \rangle.
$$
By the facts that (i) $L$ grows uniformly super-linearly (by {\rm (L3)}), and (ii) $(q,p)$ is bounded as it belongs to the subdifferential of the locally Lipschitz function  $u$, i.e. $\max\{|q|, |p|\} \leq [u]_{\Lip(\overline{J \times B(0,|x|)})}$, we deduce
$$
\Theta(|v|) - C_\Theta \leq L(I(t),x,v) \leq \ell \leq [u]_{Lip(\overline{J \times B(0,|x|)})} (1 + |v|).
$$
By the choice of $r_K$ in \eqref{eq:rk}, we must have $|v| < r_K$  with $K = |x|$. 
\begin{equation}\label{eq:choice2}
\textup{i.e.} \quad  \co(\mathcal H_{t,x}) \cap \left ( \{1\}\times \mathcal E_{t,x} \right )\cap \left ( \{1\}\times B(0,r_{|x|})\times \left [m,M\right ]\right )  \neq \emptyset \, .
\end{equation}
By \eqref{eq:choice1} and \eqref{eq:choice2}, and our choice of $R_{|x|}:= \max\{1, r_{|x|} \}$, we find that
\begin{equation*}
\co \left (- \bT_{\{\epi u\}}(t,x,u(t,x))\right  ) \cap [-G(t,x)]  = \co \left(  \mathcal{H}_{t,x}\right) \cap [-G(t,x)]  \neq \emptyset  \,,
\end{equation*}
where $G(t,x) = -\left ( \{1\}\times \mathcal E_{t,x} \right )\cap \left ( \{1\}\times B(0,R_{|x|})\times \left [m,M\right ]\right )$ is a continuous set-valued map with compact convex values. 
In order to apply the viability theorem to the closed subset $S = \epi u$, it remains to check that the statement (b) of Theorem \ref{th:viability}, i.e.
$$
\overline{\rm co} \left(\mathcal{T}_{\epi u}(t,x,U)\right)  \cap G(t,x) \neq \emptyset
$$
holds for all $(t,x,U)\in \epi u$, and not only 
 for points
 $(t,x,u(t,x))$ on the graph of $u$. This is immediate, as $ \bT_{\{\epi u\}}(t,x,U)  = \R^{d+2}$ for $U>u(t,x)$. 

Finally, all the assumptions of the viability theorem are met. As a consequence, we can remove the convex closure in \eqref{eq:non empty coH2}, and thus in \eqref{eq:non empty coH}, so as to obtain: 
\begin{equation*} 
(\forall t,x) \quad   \mathcal H_{t,x} \cap \left (  \{1\}\times \mathcal E_{t,x}\right )\neq \emptyset\, .
\end{equation*}
In particular, for each $(t,x)$ there exists a  vector  $\bv(t,x)$ such that \eqref{eq:graal} holds true.  

\noindent\textbf{Step \#6: Building the backward trajectory up to the initial time.}
Now that we are able to make a small step backward at each $(t,x)$, let $\epsilon>0$ be given, and start from $(t_0,x_0)$. There exists $(s_0,v_0)$ such that 
\begin{equation*}
u(t_0,x_0) \geq  s_0 L(t,x,v_0) + u(t_0-s_0,x_0 - s_0v_0)  - \epsilon s_0\, .
\end{equation*}
By choosing $s_0$ small enough, we can even replace the right-hand-side by:
\begin{equation}\label{eq:single step}
u(t_0,x_0) \geq \int_{0}^{s_0} L(t-s,x_0-s v_0,v_0) \, ds + u(t_0-s_0,x_0 - s_0v_0) - 2 \epsilon s_0\, .
\end{equation}
In particular, we have
\begin{equation*}
u(t_0,x_0) \geq \inf_{\gamma} \int_{t_0 - s_0}^{t_0} L(s',\gamma(s'),\dot \gamma(s'))\, ds' + u(t_0 - s_0,\gamma(t_0 - s_0)) - 2 \epsilon s_0\, , 
\end{equation*}
where the infimum is taken over all $\gamma \in AC[0,t]$ such that $\gamma(t_0) = x_0$. As a result, the set  
\begin{equation*}
\Sigma = \left \{ \tau \in(0, t_0) :  \,\, u(t_0,x_0) \geq \inf_{\gamma} \int_{\tau}^{t_0} L(s',\gamma(s'),\dot \gamma(s'))\, ds' + u(\tau,\gamma(\tau)) - 2 \epsilon (t_0-\tau) \right \}\, .
\end{equation*} 
is non-empty, and $\tau_*:= \inf \Sigma \in [0, t_0 - s_0]$ is well-defined.
We wish to prove that $\tau_* = 0$. Suppose, for contradiction, that $\tau_* >0$. Then, the estimates obtained in Lemma \ref{lem:regularity} allows to extract a converging sequence $\{\gamma_n\}$ such that $\gamma_n$ is defined on the time span $(\tau_n,t)$, with $\tau_n\searrow \tau_*$, and $\{\dot \gamma_n\}$ is uniformly $BV$ ${}^[$\footnote{The family $\{\dot\gamma_n\}$ is even uniformly Lipschitz as $I$ is assumed to be Lipschitz here, see the proof of Lemma \ref{lem:regularity}.}${}^]$. Hence, we can pass to the limit $\dot\gamma_n \to \dot\gamma$ a.e. by Helly's Selection Theorem, and then use Bounded Convergence Theorem to prove that 
\begin{equation*}
u(t_0,x_0) \geq \inf \int_{\tau_*}^{t_0} L(s',\gamma(s'),\dot \gamma(s'))\, ds' + u(\tau_*,\gamma(\tau_*)) - 2 \epsilon (t_0-\tau_*)\, .
\end{equation*}
By applying again the single step backward as in \eqref{eq:single step} at $(\tau_*,\gamma(\tau_*))$, we can push our lower estimate to an earlier time $\tau_{**} \in (0,\tau_*)$, and obtain thoroughly a contradiction. Thus, $\tau_* = 0$ and  
$$
u(t_0,x_0) \geq \inf_{\gamma} \int_0^{t_0} L(s', \gamma(s'), \dot\gamma(s'))\,ds' + g(\gamma(0)) - 2\epsilon t_0 = V(t_0,x_0) - 2\epsilon t_0.
$$
By letting $\epsilon\to 0$, we have established that $u\geq V$ in the case when $t\mapsto I(t)$ is Lipschitz continuous. 

To conclude, it remains to remove the additional continuity assumption on $I(t)$. Let $I \in BV(0,T)$.  First of all, we approximate $I(t)$ from below by a sequence of Lipschitz functions $I_k(t)\nearrow I(t)$ converging pointwise ${}^[$\footnote{Here, we choose the lower semi-continuous representative of $I$ without loss of generality. Note   that the criterion \eqref{eq:supersol} is insensitive to the choice of the representative.}${}^]$:
\begin{equation*}
I_k(t) = \inf_{s>0} \left ( I(s) + k |t-s|\right ) \leq I(t)\, .
\end{equation*}
It follows from {\rm (H2)}  
and \eqref{eq:supersol}  that $u$ is also a super-solution associated with $I_k(t)$. Hence we have 
\begin{equation} \label{eq:approx_ineq}
u \geq V_k  \quad \text{ in }(0,T) \times \mathbb{R}^d, 
\end{equation}where $V_k$ is the variational solution associated with $I_k$.

On the other hand, the compactness estimates on minimizing curves obtained in Lemma \ref{lem:regularity} combined with Lebesgue's dominated convergence theorem guarantees that 
$V_k\nearrow V$. Thus, we may let $k \to \infty$ in \eqref{eq:approx_ineq} to deduce $u \geq V$. This completes the proof.
\end{proof}

\end{document}